\documentclass[12pt]{amsart}
\usepackage[margin=1.2in]{geometry}
\usepackage{graphicx,latexsym}
\usepackage{comment}
\usepackage{tikz}
\usepackage{amsfonts, amssymb, amsmath, amsthm, bm, hyperref}
\usepackage{cite}
\usetikzlibrary{matrix,arrows}
\allowdisplaybreaks

\newcommand{\N}{\mathbb{N}}
\newcommand{\Z}{\mathbb{Z}}
\newcommand{\R}{\mathbb{R}}
\newcommand{\C}{\mathbb{C}}

\newcommand{\beq}{\begin{eqnarray}}
\newcommand{\eeq}{\end{eqnarray}}

\newcommand{\beqs}{\begin{eqnarray*}}
\newcommand{\eeqs}{\end{eqnarray*}}

\newtheorem{theorem}{Theorem}[section]
\newtheorem{proposition}[theorem]{Proposition}
\newtheorem{lemma}[theorem]{Lemma}
\newtheorem{corollary}[theorem]{Corollary}

\theoremstyle{definition}

\newtheorem{remark}[theorem]{Remark}

\newtheorem{problem}[theorem]{Problem}

\numberwithin{equation}{section}

\begin{document}
\title[Spaces of entire functions with rapid decay on strips]{Sequence space representations for spaces of entire functions with rapid decay on strips}
\author[A. Debrouwere]{Andreas Debrouwere}
\thanks{A. Debrouwere was supported by  FWO-Vlaanderen through the postdoctoral grant 12T0519N}
\address{Department of Mathematics: Analysis, Logic and Discrete Mathematics\\ Ghent University\\ Krijgslaan 281\\ 9000 Gent\\ Belgium}
\email{andreas.debrouwere@UGent.be}


\begin{abstract}
We obtain sequence space representations for a class of Fr\'echet spaces of entire functions with rapid decay on horizontal strips. In particular, we show that the projective Gelfand-Shilov spaces $\Sigma^1_\nu$ and $\Sigma^\nu_1$ are isomorphic to $\Lambda_{\infty}(n^{1/(\nu+1)})$ for $\nu > 0$. 
\end{abstract}
\maketitle

\section{Introduction}
The representation of (generalized) function spaces by sequence spaces is a classical and well-studied topic in functional analysis, see e.g.\ \cite{Bargetz, C-G-P-R, Langenbruch1987, Langenbruch1988, Langenbruch2012, Langenbruch2016, O-W, Valdivia1978, Valdivia1981, Vogt1983}. Apart from their inherent significance, such representations are important in connection with the existence of Schauder bases and the problem of isomorphic classification.

Langenbruch  \cite{Langenbruch2012,Langenbruch2016} obtained sequence space representations for a class of weighted $(LB)$-spaces of analytic germs defined on strips near $\R$. More precisely, let $\omega : [0,\infty) \rightarrow [0,\infty)$ be an increasing continuous function with $\omega(0) = 0$ and $\log t = o(\omega(t))$. Assume that $\omega$ satisfies  
\begin{equation}
\exists C, \mu > 0 \, \forall t,s \geq 0 \, : \, \omega(t+s) \leq Ce^{\mu s}(\omega(t)+1).
\label{translation-invariant}
\end{equation}
The above condition is equivalent to $\omega(t+1) = O (\omega(t))$. For $h > 0$ we write $V_h = \R + i(-h,h)$.  Consider the $(LB)$-space 
$$
\mathcal{H}_{\omega}(\R) := \bigcup_{n \in \Z_+} \{ \varphi \in \mathcal{O}(V_{1/n}) \, | \, \sup_{z \in V_{1/n}} |\varphi(z)| e^{\omega(|\operatorname{Re} z|)/n} < \infty \}.
$$
Such spaces generalize the test function space of the Fourier hyperfunctions \cite{Kawai}. In \cite[Theorem 4.4]{Langenbruch2012} it is shown that $\mathcal{H}_{\omega}(\R)$ is (tamely) isomorphic to the strong dual of some power series space $\Lambda_0(\beta)$ of finite type. Later on, Langenbruch \cite[Theorem 4.6]{Langenbruch2016}  proved that up to $O$-equivalence the sequence  $\beta$ is given by $(\omega^*(n))_{n \in \N}$, where $\omega^* = (s\omega^{-1}(s))^{-1}$. Hence, $\mathcal{H}_{\omega}(\R) \cong (\Lambda_0(\omega^*(n)))'_b$. To this end, he calculated the diametral dimension of $(\mathcal{H}_{\omega}(\R))'_b$. This result implies that the Gelfand-Shilov spaces $\mathcal{S}^1_\nu$ and $\mathcal{S}^\nu_1$ \cite{GS} are isomorphic to $(\Lambda_{0}(n^{1/(\nu+1)}))'_b$ for $\nu > 0$ \cite[Example 5.4]{Langenbruch2016}.

In the present article we obtain similar results for the Fr\'echet spaces 
$$
\mathcal{U}_{\omega}(\C) :=  \{ \varphi \in \mathcal{O}(\C) \, | \, \sup_{z \in V_{n}} |\varphi(z)| e^{n\omega(|\operatorname{Re} z|)} < \infty, \, \forall n \in \N \},
$$
which are  the projective counterparts of the spaces $\mathcal{H}_{\omega}(\R)$. Such spaces generalize the test function space of the Fourier ultrahyperfunctions \cite{PM}. Consider the following two conditions on $\omega$:
\begin{equation}
\forall \mu > 0 \, \exists C > 0 \, \forall t,s \geq 0 \, : \, \omega(t+s) \leq Ce^{\mu s}(\omega(t)+1),
\label{general}
\end{equation}
and
\begin{equation}
\exists C > 0 \, \forall t,s \geq 0 \, : \,\omega(t+s) \leq C(\omega(t) + \omega(s) + 1).
\label{general-alpha}
\end{equation}
Condition \eqref{general-alpha} is equivalent to $\omega(2t) = O(\omega(t))$, which is the well-known condition $(\alpha)$ from \cite{BMT}. Clearly, \eqref{general-alpha} implies \eqref{general}. The two main results of this article may now be formulated as follows.
\begin{theorem}\label{main-intro} \mbox{}
\begin{itemize}
\item[$(i)$] If $\omega$ satisfies \eqref{general}, then $\mathcal{U}_{\omega}(\C)$ is isomorphic to some power series space of infinite type.
\item[$(ii)$] If $\omega$ satisfies \eqref{general-alpha}, then $\mathcal{U}_{\omega}(\C) \cong \Lambda_\infty(\omega^*(n))$.
\end{itemize}
\end{theorem}
Theorem \ref{main-intro}$(ii)$ implies  that the projective Gelfand-Shilov spaces $\Sigma^1_\nu$ and $\Sigma^\nu_1$, considered e.g.\ in \cite{ACT, C-G-P-R, DV, PPV}, are isomorphic to $\Lambda_{\infty}(n^{1/(\nu+1)})$ for $\nu > 0$ (see Theorem \ref{main-example}). 

We now briefly comment on the methods used in this article.  A recent deep result of Dronov and Kaplitskii  \cite{D-K} and classical facts about complemented subspaces of $s$ \cite{Bessaga, MV}  imply that every nuclear Fr\'echet space satisfying the linear topological invariants $(\Omega)$ and $(DN)$ \cite{MV} is isomorphic to some power series space of infinite type (see Theorem \ref{basis}).
Hence, to prove Theorem \ref{main}$(i)$, it suffices to show that $\mathcal{U}_{\omega}(\C)$ satisfies $(\Omega)$ and $(DN)$. This is done in Section \ref{sect-inv} by employing techniques similar to those used in  \cite{Langenbruch2012}.  We mention that our work considerably improves the main result from \cite[Section 4]{Kruse}, where the condition $(\Omega)$ for  $\mathcal{U}_{\omega}(\C)$ is shown under much more complicated conditions on $\omega$. We also show condition $(\Omega)$ 
for the Fr\'echet spaces
$$
\mathcal{U}_{\omega}(V_h) :=  \{ \varphi \in \mathcal{O}(V_h) \, | \, \sup_{z \in V_{k}} |\varphi(z)| e^{n\omega(|\operatorname{Re} z|)} < \infty, \quad \forall k <h,  n \in \N \}, \qquad h > 0.
$$
The latter result will be used by the author in a forthcoming paper on the Borel-Ritt problem \cite{D}.  Moreover, by \cite[Theorem 5]{Kruse},  it might be useful to study the surjectivity of the Cauchy-Riemann operator on certain weighted spaces of vector-valued smooth functions. The proof of Theorem \ref{main-intro}$(ii)$ is given in Section \ref{ssr}. By Theorem \ref{main-intro}$(i)$, it is enough to show that the diametral dimension of $\mathcal{U}_{\omega}(\C)$ is equal to $\Lambda'_\infty(\omega^*(n))$. This is achieved by combining a result from \cite{Langenbruch2016} on the (generalized) diametral dimension of a  certain space $\mathcal{A}'_\omega(\R)$ of Fourier hyperfunctions of fast decay with the mapping properties of the short-time Fourier transform  on $\mathcal{U}_{\omega}(\C)$ and $\mathcal{A}'_\omega(\R)$.  
 
We are much indebted to the work of Langenbruch. The methods and results from \cite[Section 2]{Langenbruch2012}  and \cite[Section 4]{Langenbruch2016}, which are essential for the present paper,  were truly inspiring to us. 
\section{Spaces of holomorphic functions with rapid decay on strips}
In this preliminary section we introduce the  weighted Fr\'echet spaces of holomorphic functions we shall be concerned with in  this article. 

By a \emph{weight function} we mean an increasing continuous function $\omega: [0,\infty) \rightarrow [0,\infty)$ with $\omega(0) = 0$ and $\log t = o(\omega(t))$. In particular, $\omega$ is bijective on $[0,\infty)$. We extend $\omega$ to $\R$ as the even function $\omega(x) = \omega(|x|)$, $x \in \R$.  

Let $h \in (0, \infty]$. A weight function $\omega$ is said to  satisfy condition $(A)_h$ if for all $\mu > \pi/(2h)$ there is $C >0$ such that
$$
\int^\infty_0\omega(t+s)  e^{-\mu s}ds \leq C(\omega(t) + 1), \qquad t \geq 0.
$$
\begin{lemma}\label{sup-bounds}
Let $h \in (0, \infty]$ and let $\omega$ be a weight function. Then, $\omega$ satisfies $(A)_h$ if and only if for all $\mu > \pi/(2h)$ there is $C  >0$ such that
\begin{equation}
\omega(t+s) \leq Ce^{\mu s}(\omega(t)+1), \qquad t,s \geq 0.
\label{pointwise}
\end{equation}
\end{lemma}
\begin{proof} The reverse implication is clear. We now show the direct one. Let $\mu > \pi/(2h)$ be arbitrary. Define
$$
\omega_0(t) = \int^\infty_0\omega(t+s)  e^{-\mu s}ds, \qquad t \geq 0.
$$
Note that $\omega_0(t+s) \leq e^{\mu s}\omega_0(t)$ for all $t,s \geq 0$.  Since $\omega$ is increasing, we have that
$$
\omega(t) = \mu \omega(t) \int_0^\infty e^{-\mu s} ds \leq \mu \omega_0(t) \qquad t \geq 0,
$$
whereas $(A)_h$ asserts that there is $C > 0$ such that $\omega_0(t) \leq C(\omega(t)+1)$ for all $t \geq 0$. Hence,
$$
\omega(t+s) \leq \mu \omega_0(t+s) \leq \mu e^{\mu s}\omega_0(t) \leq  \mu e^{\mu  s}(\omega(t) +1), \qquad  t,s \geq 0.
$$
\end{proof}
\begin{remark}\label{remark-ti}
Lemma \ref{sup-bounds} implies that $\omega$ satisfies \eqref{general}  if and only if $\omega$ satisfies $(A)_\infty$. Similarly, $\omega$ satisfies \eqref{translation-invariant} if and only if $\omega$ satisfies $(A)_h$ for some $h > 0$.
\end{remark}
Following \cite{BMT}, a weight function $\omega$ is said to satisfy condition $(\alpha)$ if $\omega$ satisfies \eqref{general-alpha}.
Clearly, $(\alpha)$ implies $(A)_\infty$. On the contrary, the weight functions $\omega(t) = e^{(\log t)^\nu}$, $\nu > 1$, satisfy $(A)_\infty$ but not $(\alpha)$.

We write $z = x +iy \in \C$ for a complex variable. For $h \in (0,\infty]$ we set
$$
V_h := \{ z \in \C \, | \,  |y| < h \}.
$$ 
Of course, $V_\infty = \C$. Given an open set $V \subseteq \C$, we denote by $\mathcal{O}(V)$ the space of holomorphic functions in $V$.  Let $\omega$ be a weight function.  For $k > 0$ and $\lambda \in \R$ we define $\mathcal{A}_{\omega, \lambda}(V_k)$ as  the Banach space consisting of all $\varphi \in \mathcal{O}(V_k)$ such that
$$
\| \varphi \|_{k,\lambda}:= \sup_{z \in V_k} |\varphi(z)| e^{\lambda\omega(x)} < \infty.
$$
For $h \in (0,\infty]$ we define the Fr\'echet space
$$
\mathcal{U}_{\omega}(V_h) := \varprojlim_{\lambda \to \infty} \varprojlim_{k \to h^-} \mathcal{A}_{\omega,\lambda} (V_k).
$$
For $h = \infty$ we write $ \mathcal{U}_{\omega}(\C) = \mathcal{U}_{\omega}(V_\infty)$. We shall be primarily interested in this space. The condition $\log t = o(\omega(t))$ implies that ${U}_{\omega}(\C)$ is nuclear. 

\begin{remark}
In \cite[Theorem 7.2]{DV} it is shown that the space $\mathcal{U}_{\omega}(\C)$ ($\mathcal{H}_{\omega}(\R)$) is non-trivial if and only if 
$$
\int_{0}^\infty \omega(t)e^{-\mu t} dt < \infty
$$
for all $\mu > 0$ (for some $\mu > 0$). By Remark \eqref{remark-ti}, the conditions \eqref{translation-invariant} and \eqref{general} (= $(A)_\infty$) may be seen as regular versions of these non-triviality conditions (cf.\ standard non-quasianalyticity versus strong non-quasianalyticity). 
\end{remark}

\section{The conditions $(\Omega)$ and $(DN)$}\label{sect-inv}
In this section we discuss the conditions $(\Omega)$ and $(DN)$ for the Fr\'echet spaces  $\mathcal{U}_{\omega}(\C)$ and prove Theorem \ref{main-intro}$(i)$.
\subsection{The condition $(\Omega)$}  A Fr\'echet space with a fundamental decreasing sequence $(U_n)_{n \in \N}$ of neighbourhoods of zero  is said to satisfy $(\Omega)$ \cite{MV}  if
$$
\forall n \in \N \, \exists m \geq n \, \forall k \geq m \, \exists C,L > 0 \,\forall r > 0 \, :  \, U_m \subseteq  e^{-r}U_n + Ce^{Lr} U_k.
$$
\begin{proposition}\label{Omega}
Let $h \in (0,\infty]$ and let $\omega$ be a weight function satisfying $(A)_h$. Then, $\mathcal{U}_\omega(V_h)$ satisfies $(\Omega)$.
\end{proposition}
Proposition \ref{Omega} will be shown by suitably modifying the proof of  \cite[Theorem 2.2]{Langenbruch2012}. We start with the construction of holomorphic cut-off functions  \cite[p.\ 226]{Langenbruch2012}.  For $r > 0$ we define
$$
H_r(z) := \frac{1}{D_r} \int_{\gamma_z} \cosh \xi e^{-r\cosh \xi} d\xi, \qquad z \in V_{\pi/2},
$$
where $D_r = \int_{-\infty}^\infty \cosh \xi e^{-r\cosh \xi} d \xi$ and $\gamma_z$ is a path in $V_{\pi/2}$ from $-\infty$ to $z$. Let $h > 0$. For $r, A > 0$ we define
$$
E^h_{r,A}(z) := H_r\left(\frac{\pi}{2h}\left(A+z \right)\right)H_r\left(\frac{\pi}{2h}\left(A-z \right)\right),  \qquad z \in V_{h}.
$$
\begin{proposition}\label{cut-off} (cf.\ \cite[Lemma 2.3]{Langenbruch2012})
Let $h > 0$. The functions $E^h_{r,A}$, $r,A > 0$, belong to $\mathcal{O}(V_{h})$ and satisfy the following properties: For all  $0 < k <h$ there are $B,C,L,l >0$ such that for all $r ,A > 0$
\begin{gather}
|E^h_{r,A}(z)| \leq Ce^{Lr}, \qquad z \in V_k. \label{P1}\\ 
|E^h_{r,A}(z)| \leq  Ce^{-lre^{\frac{\pi}{2h}(|x | -  A)}}, \qquad z \in V_k,  |x| \geq  A + B. \label{P2} \\
|1- E^h_{r,A}(z)| \leq  Ce^{-lre^{\frac{\pi}{2h}(A - |x|)}}, \qquad z \in V_k, |x| \leq  A - B. \label{P3} 
\end{gather}
\end{proposition}
\begin{proof}
It suffices to consider the case $h = \pi/2$. The above properties follow from an inspection of the proof of \cite[Lemma 2.3]{Langenbruch2012} (and replacing $V_1$ by $V_{\pi/2}$ there) but we repeat the argument here for the sake of completeness. We claim that the functions $H_{r}$, $r> 0$, belong to $\mathcal{O}(V_{\pi/2})$ and satisfy the following properties: For all  $0 < k <\pi/2$ there are $B_0,C_0,L_0,l_0 >0$ such that for all $r > 0$
\begin{gather}
|H_r(z)| \leq C_0e^{L_0r}, \qquad z \in V_k. \label{P11}\\ 
|H_r(z)| \leq C_0e^{-l_0re^{|x|}}, \qquad z \in V_k, x \leq - B_0. \label{P22} 
\end{gather}
Furthermore, it holds that
\begin{equation}
  \label{P33}
1- H_r(z) = H_r(-z) \qquad z \in V_{\pi/2}. 
\end{equation}
Before we prove these claims, let us show how they entail the result. Property \eqref{P1} follows directly from \eqref{P11}, while \eqref{P2} is a consequence of  \eqref{P11} and \eqref{P22}. By \eqref{P33},  it holds that
$$
1- E^{\pi/2}_{r,A}(z) = H_r(-A-z)H_r(A-z) + H_r(z-A), \qquad z \in V_{\pi/2}.
$$
Property \eqref{P3} therefore also follows from  \eqref{P11} and \eqref{P22}. We now prove the claims. Since
$$
\cosh(x+iy) = \cosh x \cos y + i \sinh x \sin y, \qquad x,y \in \R,
$$
we have that 
$$
|e^{\cosh(x+iy)}| = e^{\cosh x \cos y}, \qquad x,y \in \R, 
$$
and
$$
| \cosh(x+iy)|^2 = \cosh^2x - \sin^2y, \qquad x,y \in \R.
$$
The latter equality yields that 
$$
| \cosh(x+iy)| \leq |\cosh x| , \qquad x,y \in \R.
$$
Hence, the integral defining $H_r$ is convergent on $V_{\pi/2}$. Consequently, $H_r$ is well-defined (by Cauchy's integral theorem) and holomorphic on $V_{\pi/2}$. Note hat
$$
D_r = 2 \int_0^\infty \cosh \xi e^{-r\cosh \xi} d\xi \geq 2 \int_0^\infty \sinh \xi e^{-r\cosh \xi} d\xi  = \frac{2}{re^r}.
$$
Let  $0 < k <\pi/2$ be arbitrary.  For $z  \in V_k$ we have that
\begin{align*}
|H_r(z)| &= \frac{1}{D_r} \left | \int_{-\infty}^x \cosh(\xi+iy) e^{-r\cosh(\xi+iy)} d\xi \right | \leq  \frac{1}{D_r} \int_{-\infty}^\infty \cosh\xi e^{-r\cosh \xi \cos k} d\xi   \\
&\leq  \frac{1}{D_r} \int_{|\xi| \leq 1} \cosh\xi e^{-r\cosh \xi \cos k} d\xi + \frac{2}{D_r} \int_{|\xi| \geq 1}  \sinh\xi e^{-r\cosh \xi \cos k} d\xi \\
&\leq \left( 1 + \frac{4e^{-r\cosh 1}}{D_r r\cos k}\right) e^{r\cosh 1(1-\cos k)} \leq \left( 1 + \frac{2}{\cos k}\right) e^{r\cosh 1(1-\cos k)}.
\end{align*}
This shows $\eqref{P11}$ for suitable $C_0,L_0$.  For $z \in V_k$ with $x < -1$ we have that
\begin{align*}
|H_r(z)| &= \frac{1}{D_r} \left | \int_{-\infty}^x \cosh(\xi+iy) e^{-r\cosh(\xi+iy)} d\xi \right | = \frac{1}{D_r} \left | \int_{|x|}^\infty \cosh(\xi-iy) e^{-r\cosh(\xi-iy)} d\xi \right | \\
& \leq  \frac{1}{D_r}  \int_{|x|}^\infty \cosh \xi e^{-r\cosh \xi \cos k} d\xi  
\leq \frac{2}{D_r}  \int_{|x|}^\infty \sinh \xi e^{-r\cosh \xi \cos k} d\xi  \\
&= \frac{2}{D_r r \cos k} e^{-r\cosh|x| \cos k}  
\leq  \frac{1}{ \cos k} e^{-r(\cosh|x|\cos k  - 1)}.
\end{align*} 
Since
$$
\cosh|x|\cos k  - 1 \geq \frac{\cos k  }{2} e^{|x|} - 1 \geq \frac{\cos k}{4}e^{|x|}, \qquad |x| \geq \log \left ( \frac{4}{\cos k} \right),
$$
we obtain that
$$
|H_r(z)| \leq \frac{1}{\cos k} e^{- \frac{\cos k}{4} re^{|x|}}, \qquad z \in V_k, x < - \log \left ( \frac{4}{\cos k} \right).
$$
This shows \eqref{P22} for suitable $B_0,C_0, l_0$. Finally, we show \eqref{P33}.  By Cauchy's integral formula, we have that 
\begin{align*}
1- H_r(z) &= \frac{1}{D_r} \int_{-\infty}^\infty \cosh \xi e^{-r\cosh \xi} d\xi - \frac{1}{D_r} \int_{-\infty}^x \cosh (\xi+iy) e^{-r\cosh (\xi+iy)} d\xi \\
&= \frac{1}{D_r} \int_{x}^\infty \cosh (\xi+iy) e^{-r\cosh (\xi+iy)} d\xi \\
&=   \frac{1}{D_r} \int_{-\infty}^{-x} \cosh (\xi-iy) e^{-r\cosh (\xi-iy)} d\xi = H_r(-z), \qquad  z \in V_{\pi/2}.
\end{align*}
\end{proof}
Next, we show a decomposition result for holomorphic functions on strips.
\begin{proposition} \label{decomp}(cf.\ \cite[Corollary 2.6]{Langenbruch2012})
Let $h \in (0,\infty]$ and let $\omega$ be a weight function satisfying $(A)_h$. Let $0 < k_0 < k_1 < k_2 < h$. There are $K,L > 0$ such that for all $\lambda >0$ there is $C > 0$ such that the following property holds:   
 For all  $\varphi \in \mathcal{A}_{\omega,K\lambda}(V_{k_1})$ with $ \| \varphi \|_{k_1,K\lambda} \leq 1$  and $r > 0$ there are $\varphi^r_0 \in \mathcal{A}_{\omega,\lambda}(V_{k_0})$ and $\varphi^r_2 \in \mathcal{A}_{\omega,\lambda}(V_{k_2})$ with $\varphi = \varphi^r_0 + \varphi^r_2$  such that
$$
 \| \varphi^r_0 \|_{k_0,\lambda} \leq e^{-r}\qquad \mbox{and} \qquad
\| \varphi^r_2 \|_{k_2,\lambda} \ \leq Ce^{Lr}.
$$
\end{proposition}
The proof of Proposition \ref{decomp} is based on the following consequence of H\"ormander's solution to the weighted $\overline{\partial}$-problem. 
\begin{lemma} \label{Hormander} \cite[Lemma 2.5]{Langenbruch2012} Let $0 < k_0 < k_1 < k_2$ and let $L > k_2/(k_1-k_0)$. There is $C > 0$ such that for all subharmonic functions $\psi : V_{k_2} \rightarrow \R$ the following property holds:  For all $\varphi \in \mathcal{O}(V_{k_1})$ with
$$
\left(\int_{V_{k_1}} |\varphi(z)|^2 e^{-2\psi(z)} dz \right)^{1/2} \leq 1
$$
and  $r > 0$ there are $\varphi^r_0 \in \mathcal{O}(V_{k_0})$ and $\varphi^r_2 \in \mathcal{O}(V_{k_2})$ with $\varphi = \varphi^r_0 + \varphi^r_2$  such that
$$
\left(\int_{V_{k_0}} \frac{|\varphi^r_0(z)|^2 e^{-2\psi(z)}}{(1+|z|^2)^2} dz \right)^{1/2} \leq e^{-r}\qquad \mbox{and} \qquad
\left(\int_{V_{k_2}} \frac{|\varphi^r_2(z)|^2 e^{-2\psi(z)}}{(1+|z|^2)^2} dz \right)^{1/2} \leq Ce^{Lr}.
$$
\end{lemma}
In order to be able to apply Lemma \ref{Hormander} to show Proposition \ref{decomp}, we need the following two lemmas.
\begin{lemma}\label{poisson}
Let $h \in (0,\infty]$ and let $\omega$ be a weight function satisfying $(A)_h$. For all $0 < k < h$ there exist a harmonic function $U : V_k \rightarrow \R$ and $L > 0$ such that
\begin{equation}
\frac{1}{L}\omega(x)  \leq U(z) \leq  L(\omega(x) + 1), \qquad z \in V_k.
\label{U-bounds}
\end{equation}
\end{lemma}
\begin{proof}
The Poisson kernel $P$ of the strip $V_{\pi/2}$ is given by 
$$
P(z) = \frac{\cos y}{ \cosh x + \sin y}, \qquad z \in V_{\pi/2}.
$$
The function $P$ satisfies the following properties (cf.\ \cite{DeWidder}):
\begin{itemize} 
\item[$(i)$] $P$ is positive, harmonic, and even with respect to $x$ on $V_{\pi/2}$.
\item[$(ii)$] For each $k < \pi/2$ there is $C_k > 0$ such that  $|P(z)| \leq C_ke^{-|x|}$ for all $z \in V_{k}$.
\item[$(iii)$] $\displaystyle \int_{0}^\infty P(x,y) dx = \frac{\pi}{2} - y$ for all $|y| < \pi/2$.
\end{itemize}
Let $0 < k < h$ be arbitrary. Fix $k < l <h$. Condition $(A)_h$ yields that there is $L_0 >0$ such that
\begin{equation}
\int^\infty_0\omega(x+t)  e^{-\frac{\pi}{2l} t}dt \leq L_0(\omega(x) + 1), \qquad x \geq 0.
\label{Ak}
\end{equation}
We define
$$
U(z) = \int_{-\infty}^{\infty} P \left( \frac{\pi}{2l}t, \frac{\pi}{2l}y \right) \omega(x+t) dt, \qquad z \in V_k. 
$$
Since $P$ is harmonic on $V_{\pi/2}$, $(ii)$, and \eqref{Ak} imply that  $U$ is harmonic on $V_k$. As $U$ is even with respect to $x$, it suffices to show \eqref{U-bounds} for all $z \in V_k$ with $x \geq 0$.  Condition $(iii)$ implies that
\begin{align*}
U(z) &\geq \int_0^\infty P \left( \frac{\pi}{2l}t, \frac{\pi}{2l}y \right) \omega(x+t) dt \geq \omega(x) \int_0^\infty P \left( \frac{\pi}{2l}t, \frac{\pi}{2l}y \right)  dt \\
&=  \omega(x) \frac{2l}{\pi} \int_{0}^\infty P \left( t, \frac{\pi}{2l}y \right)  dt =  \omega(x) (l-y) \geq \omega(x) (l-k).
\end{align*}
By $(ii)$ and \eqref{Ak}, we have that
\begin{align*}
U(z) &\leq 2\int_{0}^{\infty} P \left( \frac{\pi}{2l}t, \frac{\pi}{2l}y \right) \omega(x+t) dt \leq 2C_{\pi k /(2l)}\int_{0}^{\infty} \omega(x+t) e^{-\frac{\pi}{2l} t}dt \\
&\leq 2C_{\pi k /(2l)}L_0(\omega(x) + 1).
\end{align*}
This shows \eqref{U-bounds} for suitable $L$.
\end{proof}
Next, we compare $\sup$-norms with $L^2$-norms. Let $\omega$ be a weight function.  For $h, \lambda > 0$ we define $\mathcal{A}_{2;\lambda}(V_h)$ as the Banach space consisting of all $\varphi \in \mathcal{O}(V_h)$ such that
$$
\| \varphi\|_{2; h,\lambda} := \left(\int_{V_{h}} |\varphi(z)|^2 e^{2\lambda\omega(x)} dz \right)^{1/2} < \infty.
$$
\begin{lemma}\label{L-2} Let $\omega$ be a weight function satisfying \eqref{translation-invariant}. 
\begin{itemize}
\item[$(i)$]
For all $h > 0$ and $0 < \mu <\lambda$ there is $C > 0$ such that
$$
\| \varphi\|_{2; h,\mu} \leq C \| \varphi\|_{h,\lambda}, \qquad \varphi \in \mathcal{A}_{\lambda}(V_h).
$$
\item[$(ii)$] For all $0 < k < h$ there is $L > 0$ such that for all $\lambda > 0$ there is $C > 0$ with
$$
\| \varphi\|_{k,\lambda} \leq C \| \varphi\|_{2;h,L\lambda}, \qquad \varphi \in \mathcal{A}_{2;L\lambda}(V_h).
$$
\end{itemize}
\end{lemma}
\begin{proof}
$(i)$ We have that 
$$
\| \varphi\|_{2; h,\mu} \leq \| \varphi\|_{h,\lambda}  \left(\int_{V_{h}} e^{-2(\lambda - \mu)\omega(x)} dz \right)^{1/2}, \qquad  \varphi \in \mathcal{A}_{\lambda}(V_h).
$$
$(ii)$ Property \eqref{translation-invariant} implies that there is $L > 0$ such tat
$$
\omega(x+y) \leq L (\omega(x)  + 1), \qquad x \in \R, |y| \leq h-k.
$$
Let $\varphi \in \mathcal{A}_{2;L\lambda}(V_h)$ be arbitrary. The mean value theorem implies that
$$
\varphi(z) = \frac{1}{\pi(h-k)^2} \int_{B(z,h-k)} \varphi(w) dw, \qquad z \in V_k.
$$
Hence,
\begin{align*}
\| \varphi\|_{k,\lambda} &\leq \frac{1}{\pi(h-k)^2} \sup_{z \in V_k}  e^{\lambda\omega(x)} \int_{B(z,h-k)} |\varphi(w)| dw \\
&\leq  \frac{e^{\lambda L}}{\pi(h-k)^2}  \sup_{z \in V_k} \int_{B(z,h-k)} |\varphi(w)|e^{L\lambda \omega(\operatorname{Re} w)} dw  \\
&\leq  \frac{e^{\lambda L}}{\sqrt{\pi}(h-k)}  \sup_{z \in V_k} \left(\int_{B(z,h-k)} |\varphi(w)|^2e^{2L\lambda \omega(\operatorname{Re} w)} dw\right)^{1/2} \\
&\leq  \frac{e^{\lambda L}}{\sqrt{\pi}(h-k)}  \| \varphi\|_{2;h,L\lambda}. 
\end{align*}
\end{proof}
\begin{proof}[Proof of Proposition \ref{decomp}] This is a consequence of Lemma \ref{Hormander}, Lemma \ref{poisson}, and Lemma \ref{L-2} (cf.\ the proof of \cite[Corollary 2.6]{Langenbruch2012}).
\end{proof}
Finally, we need the following abstract lemma.
\begin{lemma}\label{omega-loc}
Let $E_0 \supset E_1 \supset \cdots$ be a decreasing sequence of Banach spaces with $\| \, \cdot \, \|_{E_0} \leq \| \, \cdot \, \|_{E_1} \leq \cdots$.  Denote by $U_n$ the unit ball in $E_n$ for  $n \in \N$. Set $E = \bigcap_{n \in \N} E_n$ and endow $E$ with its natural projective limit topology, that is, the topology generated by the basis $\{ U_n \cap E \, | \, n \in \N\}$ of neighbourhoods of zero. Suppose that
\begin{equation}
\forall n \in \N \, \exists m \geq n \, \forall k \geq m \, \exists C,L > 0 \, \forall r > 0 \, :  \, U_m \subseteq  e^{-r}U_n + Ce^{Lr} U_k.
\label{Omega-loc}
\end{equation}
Then, $E$ satisfies $(\Omega)$.
\end{lemma}
\begin{proof}
Condition \eqref{Omega-loc} yields that
$$
\forall n \in \N \, \exists \widetilde{n} \geq n \, \forall k \geq m \,  \forall \varepsilon > 0 \, :  \, E_{\widetilde{n}} \subseteq  E_k + \varepsilon U_n.
$$
Since the  spaces $E_n$ are Banach, a standard argument (see e.g.\ \cite[Lemma, p.\ 260]{Dubinsky-proj})  shows that the above property in fact entails the following a priori stronger property
$$
\forall n \in \N \, \exists \widetilde{n} \geq n \,  \forall \varepsilon > 0 \, :  \, E_{\widetilde{n}} \subseteq  E + \varepsilon U_n.
$$
Let $n \in \N$ be arbitrary. Choose $m \geq n$ according to \eqref{Omega-loc}. Let $k \geq m$ be arbitrary. Choose $\widetilde{k} \geq k$ such that $E_{\widetilde{k}} \subseteq  E + \varepsilon U_k$ for all $\varepsilon > 0$.  Choose $C$ and $L$ according to  \eqref{Omega-loc} with $\widetilde{k}$ instead of $k$. Let $e \in E$ with $\|e\|_m \leq 1$ and $r > 0$ be arbitrary. Then, there exist $\widetilde{x}^r \in E_{\widetilde{k}}$ such that 
$$
\| e-\widetilde{x}^r\|_{E_n} \leq e^{-r} \qquad \mbox{and} \qquad \|\widetilde{x}^r\|_{E_{\widetilde{k}}} \leq Ce^{Lr}.
$$
Choose $x^r \in E$ such that $\| \widetilde{x}^r  -x^r \|_{E_k} \leq e^{-r}$. Then, 
$$
\| e-x^r\|_{E_n}  \leq \| e-\widetilde{x}^r\|_{E_n}  + \|\widetilde{x}^r- x^r\|_{E_k} \leq 2e^{-r} 
$$
and
$$
\|x^r\|_{E_k} \leq \|  x^r-\widetilde{x}^r\|_{E_k} + \|\widetilde{x}^r\|_{E_{\widetilde{k}}} \leq (C+1)e^{Lr}.
$$
This shows that $E$ satisfies $(\Omega)$.
\end{proof}
\begin{proof}[Proof of Proposition \ref{Omega}] By Lemma \ref{omega-loc}, it suffices to show the following statement: \emph{Let $0 < k_0 < k_1 < k_2 < h$ and let $\lambda_0 > 0$  be arbitrary. Set $\lambda_1 = \max \{ K \lambda_0, \lambda_0+1\}$, where $K$ is the constant from Proposition \ref{decomp}. Let $\lambda _2 \geq \lambda_1$ be arbitrary. There are $C_0, C_2, L_0,L_2 > 0$ such that  for all $\varphi \in \mathcal{A}_{\lambda_1}(V_{k_1})$ with $\| \varphi\|_{k_1,\lambda_1} \leq 1$ and  $r > 0$ there exist $\psi^r_0 \in \mathcal{A}_{\omega,\lambda_0}(V_{k_0})$ and $\psi^r_2 \in \mathcal{A}_{\omega,\lambda_2}(V_{k_2})$ with $\varphi = \psi^r_0 + \psi^r_2$  such that
\begin{equation}
 \| \psi^r_0 \|_{k_0,\lambda_0} \leq C_0e^{-r/L_0}\qquad \mbox{and} \qquad
\| \psi^r_2 \|_{k_2,\lambda_2} \ \leq C_2e^{L_2r}.
\label{boundz}
\end{equation}}
We may assume without loss of generality that $r \geq 1$. Fix $k_2 < k_3 < h$. Proposition \ref{cut-off} yields that the functions $E_{r,A} = E^{k_3}_{r,A}$, $r \geq 1, A > 0$, belong to $\mathcal{O}(V_{k_2})$ and that  there are $B,C_3,L_3,l_3 >0$ with
\begin{gather*}
|E_{r,A}(z)| \leq C_3e^{L_3r}, \qquad z \in V_{k_2}. \\ 
|E_{r,A}(z)| \leq  C_3e^{-l_3r e^{\frac{\pi}{2k_3}(|x | -  A )}}, \qquad z \in V_{k_2},  |x| \geq  A + B.\\
|1- E_{r,A}(z)| \leq  C_3e^{-l_3re^{\frac{\pi}{2k_3}( A- |x |)}}, \qquad z \in V_{k_2}, |x| \leq  A - B.  
\end{gather*}
 By Proposition \ref{decomp}, there are $C_4,L_4 > 0$ such that for all $\varphi \in \mathcal{A}_{\lambda_1}(V_{k_1})$ with $\| \varphi\|_{k_1,\lambda_1} \leq 1$ and  $r \geq 1$ there exist $\varphi^r_0 \in \mathcal{A}_{\omega,\lambda_0}(V_{k_0})$ and  $\varphi^r_2 \in \mathcal{A}_{\omega,\lambda_0}(V_{k_2})$ with $\varphi = \varphi^r_0 + \varphi^r_2$  such that
$$
 \| \varphi^r_0 \|_{k_0,\lambda_0} \leq e^{-r}\qquad \mbox{and} \qquad
\| \varphi^r_2 \|_{k_2,\lambda_0} \ \leq C_4e^{L_4r}.
$$
Choose $A_r > 0$ such that $\omega(A_r - B) = r$ for $r \geq 1$.  Let $\varphi \in \mathcal{A}_{\lambda_1}(V_{k_1})$ with $\| \varphi\|_{k_1,\lambda_1} \leq 1$ and  $r \geq 1$
be arbitrary. We set
$$
\psi^r_0 = (1- E_{r/(2L_3), A_r})\varphi + E_{r/(2L_3), A_r}\varphi^r_0 \qquad \mbox{and} \qquad
\psi^r_2 = E_{r/(2L_3), A_r}\varphi^r_2.
$$
Then, $\psi^r_0 \in \mathcal{O}(V_{k_0})$ and $\psi^r_2 \in \mathcal{O}(V_{k_2})$. Moreover, it is clear that $\varphi = \psi^r_0 + \psi^r_2$. We now show \eqref{boundz}.
Fix $k_3 < k_5 < h$. By Lemma \ref{sup-bounds}, there is $L_5 > 0$ such that
$$
\omega(x+y) \leq L_5e^{\frac{\pi}{2k_5} |x|} (\omega(y) + 1), \qquad x,y \in \R.
$$
For $z \in V_{k_0}$ with $|x| \leq A_r - B$ it holds that
$$
 |(1- E_{r/(2L_3), A_r}(z))\varphi(z)|e^{\lambda_0 \omega(x)} \leq   C_3e^{-\frac{l_3}{2L_3}re^{\frac{\pi}{2k_3}( A_r- |x |)}} \leq C_3e^{-\frac{l_3}{2L_3}e^{\frac{\pi B}{2k_3}}r} .
 $$
 For $z \in V_{k_0}$ with $|x| \geq A_r - B$ it holds that
$$
 |(1- E_{r/(2L_3), A_r}(z))\varphi(z)|e^{\lambda_0 \omega(x)} \leq (C_3+1)e^{\frac{r}{2} - \omega(x)}  \leq (C_3+1)e^{\frac{r}{2} - \omega(A_r-B)} \leq (C_3+1)e^{-\frac{r}{2}}.
 $$
 For $z \in V_{k_0}$ it holds that
 $$
 |E_{r/(2L_3), A_r}(z)\varphi^r_0(z)|e^{\lambda_0 \omega(x)} \leq C_3e^{\frac{r}{2}} \| \varphi^r_0 \|_{k_0,\lambda_0} \leq C_3e^{-\frac{r}{2}}.
 $$
 Hence, $ \| \psi^r_0 \|_{k_0,\lambda_0} \leq C_0e^{-r/L_0}$ for suitable $C_0,L_0$. Next, we consider $\psi^r_2$. Note that
 $$
 \omega(A_r+B) \leq L_5 e^{\frac{\pi B}{k_5}}(\omega(A_r-B) +1) \leq  L_5 e^{\frac{\pi B}{k_5}}(r +1).
 $$
 Hence, for $z \in V_{k_2}$ with $|x| \leq A_r+B$  it holds that
 $$
  |\psi^r_2(z)| e^{\lambda_2 \omega(x)} \leq C_3 e^{\frac{r}{2} + \lambda_2 \omega(A_r + B)} \| \varphi^r_2\| _{k_2, \lambda_0} \leq C_3 C_4 e^{\lambda_2L_5e^{\frac{\pi B}{k_5}}} e^{(\frac{1}{2} + L_4 + \lambda_2L_5e^{\frac{\pi B}{k_5}} )r}.
 $$
Since $r \geq 1$, we have that
  $$
 \omega(x) \leq L_5e^{\frac{\pi}{2k_5} |x-A_r + B|} (\omega(A_r-B) + 1) \leq 2L_5 e^{\frac{\pi B}{2k_5}}e^{\frac{\pi}{2k_5} |x-A_r|} r, \qquad x \in \R.
 $$
Set $L_6 =  2L_5 e^{\frac{\pi B}{2k_5}}$. Then, for $z \in V_{k_2}$ with $|x| \geq A_r+B$ it holds that
 \begin{align*}
 |\psi^r_2(z)| e^{\lambda_2 \omega(x)} &\leq C_3e^{-\frac{l_3}{2L_3}r e^{\frac{\pi}{2k_3}(|x | -  A_r)} + \lambda_2 \omega(x)}\| \varphi^r_2\| _{k_2, \lambda_0}  \\
 &\leq C_3C_4e^{-\frac{l_3}{2L_3}r e^{\frac{\pi}{2k_3}(|x | -  A_r)} + \lambda_2L_6re^{\frac{\pi}{2k_5} (|x | -  A_r)} + L_4r} \\
 &\leq C_3C_4 e^{ \left ( \lambda_2L_6 -  \frac{l_3}{2L_3} e^{(\frac{\pi}{2k_3} -\frac{\pi}{2k_5}) (|x | -  A_r)} \right)e^{\frac{\pi}{2k_5} (|x | -  A_r)} r + L_4r}
 \end{align*}
 Set
 $$
 L_7 = \frac{\max \{ 0, \log(2\lambda_2L_3L_6/l_3) \}}{\frac{\pi}{2k_3} - \frac{\pi}{2k_5}}.
 $$
 If $|x| \geq A_r+L_7$, we have that
 $$
 |\psi^r_2(z)| e^{\lambda_2 \omega(x)} \leq C_3C_4e^{L_4r}.
 $$ 
 If  $|x| \leq A_r+L_7$, we have that
 $$
|\psi^r_2(z)| e^{\lambda_2 \omega(x)} \leq  C_3C_4  e^{(L_4 + \lambda_2L_6 e^{\frac{\pi L_6}{2k_5}}) r}.
 $$
 Hence, $\| \psi^r_2 \|_{k_2,\lambda_2} \leq C_2e^{L_2 r}$ for suitable $C_2, L_2$.
\end{proof}

\subsection{The condition $(DN)$} A Fr\'echet space $E$ with a fundamental increasing sequence $(\| \, \cdot \, \|_n)_{n \in \N}$ of seminorms  is said to satisfy $(DN)$ \cite{MV}  if
$$
\exists n \in \N \, \forall m \geq n \, \exists k \geq m \exists C > 0  \, \forall e \in E \, : \, \|e\|^2_m \leq C \|e \|_n \|e \|_k.
$$
\begin{proposition}\label{DN}
Let $\omega$ be a weight function satisfying \eqref{translation-invariant}. Then, $\mathcal{U}_\omega(\C)$ satisfies $(DN)$.
\end{proposition}
Proposition \ref{DN} is a direct consequence of the following weighted version of the Hadamard three-lines theorem; its proof is inspired by the proof of \cite[Proposition 3.1]{Langenbruch2012}.
\begin{lemma}\label{three-lines}
Let $\omega$ be a weight function satisfying \eqref{translation-invariant}. Let $0 <k_0 < k_1 < k_2 < h$. Set $\theta = \frac{\log(k_2/k_1)}{\log(k_2/k_0)}$. There is $L > 0$ such that for all $\lambda > 0$ there is $C > 0$ with
$$
\|\varphi\|_{k_1,\lambda} \leq C \|\varphi \|^{\theta}_{k_0, 0} \|\varphi \|^{1-\theta}_{k_2, L\lambda}, \qquad  \varphi \in \mathcal{U}_\omega(V_{h}).
$$
\end{lemma}
\begin{proof}
 Property \eqref{translation-invariant} implies that there is $L > 0$ such that
$$
\omega(x+y) \leq L (\omega(x)  + 1), \qquad x \in \R, |y| \leq k_2.
$$
For $k > 0$ we write $\| \, \cdot \, \|_{B(0,k)}  = \| \, \cdot \, \|_{L^\infty(B(0,k))}$. The Hadamard three-circles theorem yields that
$$
\| \varphi\|_{B(0,k_1)} \leq \| \varphi\|^\theta_{B(0,k_0)} \| \varphi\|^{1-\theta}_{B(0,k_2)}, \qquad \varphi \in \mathcal{O}(B(0,h)).
$$
Hence, for all $\varphi \in \mathcal{U}_\omega(V_{h})$,
\begin{align*}
\|\varphi\|_{k_1,\lambda} &\leq \sup_{x \in \R} \| \varphi(x + \, \cdot \, )\|_{B(0,k_1)} e^{\lambda \omega(x)} \\
&\leq  \sup_{x \in \R} \| \varphi(x + \, \cdot \, )\|^\theta_{B(0,k_0)}  \| \varphi(x + \, \cdot \, )\|^{1-\theta}_{B(0,k_2)}e^{\lambda \omega(x)} \\
&\leq  \|\varphi \|^{\theta}_{k_0, 0} \left (\sup_{x \in \R} \max_{w \in B(0,k_2)} |\varphi(x+w)|e^{\frac{\lambda}{1-\theta} \omega(x)}  \right)^{1-\theta} \\
&\leq e^{\lambda L} \|\varphi \|^{\theta}_{k_0, 0}  \|\varphi \|^{1-\theta}_{k_2, L\lambda/(1-\theta)}.
\end{align*}
\end{proof}
\subsection{Implicit sequence space  representations} 
Let $\beta = (\beta_n)_{n \in \N}$ be an increasing sequence of  positive numbers with $\lim_{n \to \infty}\beta_n = \infty$. We define  the \emph{power series space $\Lambda_\infty(\beta)$ of infinite type} as the Fr\'echet space consisting of all $(c_n)_{n \in \N} \in \C^{\N}$ such that for all $r > 0$
$$
\sup_{n \in \N} |c_n| e^{r\beta_n}  < \infty.
$$

The space $\Lambda_\infty(\beta)$ is nuclear if and only if $\log n= O(\beta_n)$ \cite[Proposition 29.6]{MV}. We set $s = \Lambda_\infty ( \log (e+n))$.
As mentioned in the introduction, the following result holds.
\begin{theorem}\label{basis}
Let $E$ be a nuclear Fr\'echet space satisfying $(\Omega)$ and $(DN)$. Then, $E$ is isomorphic to some power series space of infinite type.
\end{theorem}
\begin{proof}
By \cite[Proposition 31.7]{MV}, $E$ is isomorphic to a complemented subspace of $s$. Hence, \cite[Theorem 2, p.\ 1474]{D-K} yields that $E$ admits an (absolute) Schauder basis. The result now follows from \cite[Corollary 3.4]{Bessaga}.
\end{proof}
Proposition \ref{Omega}, Proposition \ref{DN}, and Theorem \ref{basis} yield the following result.
\begin{theorem}\label{main-1}
Let $\omega$ be a weight function satisfying $(A)_\infty$. Then, $\mathcal{U}_\omega(\C)$ is isomorphic to some power series space $\Lambda_\infty(\beta)$ of infinite type.
\end{theorem}
\section{Explicit sequence space representations}\label{ssr}
Let $\omega$ be a weight function. We define
$$
\omega^* := (s\omega^{-1}(s))^{-1}.
$$
Note that $\omega^*$ is a weight function satisfying $(\alpha)$ (even if $\omega$ itself does not satisfy $(\alpha)$). Moreover, we have that \cite[Equation $(3.1)$]{Langenbruch2016}
\begin{equation}
\omega^*(t) = \frac{t}{(s\omega(s))^{-1}(t)}, \qquad t \geq 0.
\label{l-repr}
\end{equation}
The goal of this section is to show the following result.
\begin{theorem}\label{main}
Let $\omega$ be a weight function satisfying $(\alpha)$. Then, $\mathcal{U}_\omega(\C) \cong \Lambda_\infty(\omega^*(n))$.
\end{theorem}
We need various results in preparation for the proof of Theorem \ref{main}.
\subsection{The short-time Fourier transform}\label{STFT}
The \emph{short-time Fourier transform} (STFT) of $f \in L^{2}(\mathbb{R})$ with respect to the window $\psi \in L^{2}(\mathbb{R})$ is given by
	\[ V_{\psi} f(x, \xi) := \int_{\mathbb{R}} f(t) \overline{\psi(t - x)} e^{- 2 \pi i \xi  t} dt, \qquad (x, \xi) \in \mathbb{R}^{2}. \]
The adjoint of the bounded linear mapping $V_{\psi}: L^{2}(\mathbb{R})\to L^{2}(\mathbb{R}^2)$ is given by the weak integral
	\[ V_{\psi}^{*} F(t) =  \iint _{\mathbb{R}^{2}} F(x, \xi) e^{2 \pi i \xi  t}\psi(t-x) dx d\xi, \qquad F \in L^2(\R^2) . \]
Furthermore, if $\gamma \in L^2(\R)$ is such $(\gamma, \psi)_{L^{2}} \neq 0$, then the reconstruction formula
	\begin{equation} 
		\label{eq:reconstructSTFT}
		\frac{1}{(\gamma, \psi)_{L^{2}}} V_{\gamma}^{*} \circ V_{\psi} = {\operatorname{id}}_{L^{2}(\mathbb{R})} 
	\end{equation}
holds. We refer to \cite{Grochenig} for more information on the STFT.

 Our first goal is to establish the mapping properties of the STFT on  $\mathcal{U}_{\omega}(\C)$.  Let $\omega$ be a weight function. For $h, \lambda \in \R$ we define $C_{\omega,\lambda,h}(\R^2)$ as the Banach space consisting of all $f \in C(\R^2)$ such that
$$
\|f\|_{C_{\omega,\lambda,h}(\R^2)} := \sup_{(x,\xi) \in \R^2} |f(x,\xi)|e^{\lambda\omega(x) + h|\xi|} < \infty.
$$
 We define the Fr\'echet space
$$
C_{\omega}(\R^2) := \varprojlim_{h \to \infty} C_{\omega,h,h}(\R^2). 
$$
For $\varphi \in \mathcal{O}(\C)$ we define $\varphi^*(z) := \overline{\varphi(\overline{z})}$, $z \in \C$. Then, $\varphi^* \in  \mathcal{O}(\C)$ and $\varphi^*(x) = \overline{\varphi(x)}$ for $x \in \R$. If $\varphi \in \mathcal{U}_{\omega}(\C)$, then $\varphi^* \in \mathcal{U}_{\omega}(\C)$ and 
$\| \varphi^* \|_{\mathcal{A}_{\omega, \lambda} (V_h)} = \| \varphi \|_{\mathcal{A}_{\omega, \lambda} (V_h)}$ for all $h> 0$ and $\lambda \in \R$.
\begin{proposition} \label{STFT-test} Let $\omega$ be a weight function satisfying $(\alpha)$. Fix a window $\psi \in \mathcal{U}_{\omega}(\C)$.
The linear mappings 
$$
V_{\psi}: \mathcal{U}_{\omega}(\C) \rightarrow C_{\omega}(\R^2)\quad \mbox{ and } \quad V^{\ast}_{\psi}:C_{\omega}(\R^2) \to \mathcal{U}_{\omega}(\C)
$$
are continuous.
\end{proposition}
\begin{proof} Since $\omega$ satisfies $(\alpha)$, there is $L > 0$ such that
\begin{equation}
\omega(x+y) \leq L(\omega(x) + \omega(y) + 1), \qquad x,y \in \R.
\label{alphaa}
\end{equation}
We first consider $V_\psi$. Fix $h > 0$ and let $\varphi \in \mathcal{U}_\omega(\C)$ be arbitrary.  By Cauchy's integral theorem we have that
\begin{align*}
V_\psi \varphi(x,\xi) &=  \int_{\mathbb{R}} \varphi(t) \psi^*(t - x) e^{- 2 \pi i \xi  t} dt \\
&= \int_{\R} \varphi(t - i \operatorname{sgn}(\xi)h) \psi^*(t - x - i \operatorname{sgn}(\xi) h) e^{- 2 \pi i \xi  (t - i \operatorname{sgn}(\xi) h)} dt.
\end{align*}
Hence,
\begin{align*}
|V_\psi \varphi(x,\xi)| &\leq e^{-2\pi h |\xi|} \int_{\R} |\varphi(t - i \operatorname{sgn}(\xi)h)| |\psi^*(t - x - i \operatorname{sgn}(\xi) h)| dt \\
&\leq  e^{Lh}\int_{\R}e^{-Lh\omega(t)} dt \| \psi \|_{h, Lh}\| \varphi \|_{h,2Lh} e^{- h\omega(x) -2\pi h |\xi|}.
\end{align*}
This shows that $V_\psi$ is continuous. Next, we treat $V^*_\psi$. Fix $h > 0$ and  let $F \in C_\omega(\R^2)$ be arbitrary.  Note that
$$
V_{\psi}^{*} F(t+iu) =  \iint _{\mathbb{R}^{2}} F(x, \xi) e^{2 \pi i \xi  (t+iu)}\psi(t+iu-x) dx d\xi, \qquad t+iu \in \C,
$$
is an entire function. For all $t+iu \in V_h$ it holds that
\begin{align*}
|V_{\psi}^{*} F(t+iu)| &\leq \iint _{\mathbb{R}^{2}} |F(x, \xi)| e^{2 \pi h |\xi|} |\psi(t+iu-x)| dx d\xi \\
& \leq e^{Lh}\int_{\R}e^{-Lh\omega(x)} dx \int_\R e^{-\pi h |\xi|} d\xi \| \psi \|_{h,Lh}\| F \|_{\mathcal{C}_{\omega,2Lh, 3\pi h} (\R^2)} e^{- h\omega(t)}.
\end{align*}
This shows that $V^*_\psi$ is continuous. 
\end{proof}
\begin{remark}
Let $\omega$ be a weight function satisfying $(\alpha)$. Proposition \ref{STFT-test} and \eqref{eq:reconstructSTFT} imply that $\mathcal{U}_{\omega}(\C)$ is isomorphic to a complemented subspace of $C_{\omega}(\R^2)$. A standard argument involving cut-off functions shows that $C_{\omega}(\R^2)$ satisfies $(\Omega)$, while it is clear that $C_{\omega}(\R^2)$ satisfies $(DN)$. Since  $(\Omega)$ and $(DN)$ are inherited to complemented subspaces, we obtain a simple alternative proof of the fact that $\mathcal{U}_{\omega}(\C)$ satisfies $(\Omega)$ and $(DN)$ (compare with Proposition \ref{Omega} and Proposition \ref{DN}).
\end{remark}
Let $\omega$ be a weight function. We define the following $(LB)$-space
$$
\mathcal{A}_{\omega}(\R) := \varinjlim_{h \to \infty} \mathcal{A}_{\omega,-h} (V_{1/h}).
$$
Note that $\mathcal{U}_\omega(\C) \subset \mathcal{A}_{\omega}(\R)$. If $\omega$ satisfies $(\alpha)$, $\mathcal{A}_{\omega}(\R)$ coincides with the space $\mathcal{H}^\infty_{\omega}(\R)$ from \cite[p.\ 91]{Langenbruch2016}. We define $\mathcal{A}'_{\omega}(\R) = (\mathcal{A}_{\omega}(\R))'_b$. We may view $\mathcal{U}_{\omega}(\C)$ as a vector  subspace of $\mathcal{A}'_{\omega}(\R)$ by identifying $\varphi \in \mathcal{U}_{\omega}(\C)$ with the element of $\mathcal{A}'_{\omega}(\R)$ given by
$$
\langle \varphi, \psi \rangle := \int_{\R} \varphi(x) \psi(x) dx, \qquad \psi \in \mathcal{A}_{\omega}(\R).
$$
The inclusion mapping $\mathcal{U}_{\omega}(\C) \rightarrow \mathcal{A}'_{\omega}(\R)$ is continuous. We now extend the STFT and its adjoint to $\mathcal{A}'_{\omega}(\R)$. To this end, we define the Fr\'echet space
$$
\widetilde{C}_{\omega}(\R^2) := \varprojlim_{h \to \infty} C_{\omega,h,-1/h}(\R^2),
$$
The STFT of  $f \in \mathcal{A}'_{\omega}(\R)$ with respect to the window $\psi \in \mathcal{U}_{\omega}(\C)$ is defined as
$$
V_{\psi} f(x, \xi) := \langle f(t), \psi^*(t - x) e^{- 2 \pi i \xi  t} \rangle, \qquad (x, \xi) \in \mathbb{R}^{2}. 
$$
Clearly, $V_\psi f$ is a continuous function on $\R^2$. We define the adjoint STFT of $F \in \widetilde{C}_{\omega}(\R^2) $ as
$$
\langle V^*_\psi F, \varphi \rangle :=  \iint _{\mathbb{R}^{2}} F(x, \xi) V_{\psi^*} \varphi(x,-\xi) dx d\xi, \qquad \varphi \in \mathcal{A}_{\omega}(\R).
$$
 \begin{proposition} \label{STFT-dual-2} Let $\omega$ be a weight function satisfying $(\alpha)$. Fix a window $\psi \in \mathcal{U}_{\omega}(\C)$.
The linear mappings 
$$
V_{\psi}: \mathcal{A}'_{\omega}(\R) \rightarrow \widetilde{C}_{\omega}(\R^2)\quad \mbox{ and } \quad V^{\ast}_{\psi}:\widetilde{C}_{\omega}(\R^2) \to \mathcal{A}'_{\omega}(\R)
$$
are continuous. Furthermore, if $\gamma \in \mathcal{U}_{\omega}(\C)$ is such that  $(\gamma, \psi)_{L^{2}} \neq 0$, then the reconstruction formula
\begin{equation}
\frac{1}{(\gamma, \psi)_{L^{2}}} V_{\gamma}^{*} \circ V_{\psi} = {\operatorname{id}}_{\mathcal{A}'_{\omega}(\R)} 
\label{reconstruction-dual-2}
\end{equation}
holds.
\end{proposition}
\begin{proof}
Choose $L > 0$ such that \eqref{alphaa} is satisfied. We first consider $V_\psi$. Fix $h > 0$ and let $f \in  \mathcal{A}'_{\omega}(\R)$ be arbitrary. Note that $\| V_\psi f\|_{C_{\omega,h,-1/h}(\R^2)} = \sup_{\varphi \in B_h} |\langle f, \varphi \rangle|$, where
$$
B_h = \{ \psi^*(t - x) e^{- 2 \pi i \xi  t} e^{h\omega(x) -|\xi|/h} \, | \, (x,\xi) \in \R^2 \}.
$$
The set $B_h$ is bounded in $\mathcal{A}_{\omega}(\R)$, as follows from the estimate
$$
 \| \psi^*(t - x) e^{- 2 \pi i \xi  t} \|_{1/(2\pi h), -Lh} \leq L\| \psi \|_{1/(2\pi h), Lh}  e^{-h\omega(x) +  |\xi|/h}, \qquad (x,\xi) \in \R^2.
$$
This shows that $V_\psi$ is continuous. Next, we treat $V^{\ast}_{\psi}$. It suffices to show that the linear mapping
$$
V_{\psi^*}: \mathcal{A}_{\omega}(\R) \rightarrow \varinjlim_{h \to \infty} C_{\omega,-h,1/h}(\R^2)
$$
is continuous. This can be done by using an argument similar to the one used in the first part of the proof of Proposition \ref{STFT-test}. The Hahn-Banach theorem and the fact that $\mathcal{A}_{\omega}(\R)$ is reflexive imply that $\mathcal{U}_{\omega}(\C)$ is dense in $\mathcal{A}'_{\omega}(\R)$. Hence, \eqref{reconstruction-dual-2} follows from \eqref{eq:reconstructSTFT}.
\end{proof}
\subsection{The diametral dimension}\label{sect-diam}
 Let $E$ be a vector space.  For $n \in \N$ we denote by $\mathcal{L}_n(E)$ the set consisting of all subspaces $F \subseteq E$ with $\operatorname{dim} F \leq n$. Given subsets $V \subseteq U \subseteq E$ and $n \in \N$, we set
$$
\delta_n(V,U) := \inf \{ \delta > 0 \, | \, \exists F \in \mathcal{L}_n(E) \, : \, V \subseteq \delta U + F \}.
$$
The \emph{diametral dimension} \cite{Jarchow} of a Fr\'echet space $E$ with a fundamental decreasing sequence $(U_n)_{n \in \N}$ of neighbourhoods of zero is defined  as the set
$$
\Delta(E) := \{  (c_n)_{n \in \N} \in \C^\N \, | \, \forall m \in \N \, \exists k \geq m \, : \, \lim_{n \to \infty} c_n \delta_n(U_k,U_m) = 0 \}.
$$
 We have that $\Delta(\Lambda_\infty(\beta)) = \Lambda'_\infty(\beta)$ for any nuclear power series space $\Lambda_\infty(\beta)$ of infinite type \cite[Proposition 10.6.10]{Jarchow}. 

Our aim is to  show that $\Delta(\mathcal{U}_\omega(\C)) = \Lambda'_\infty(\omega^*(n))$. We start with the inclusion $\Lambda'_\infty(\omega^*(n)) \subseteq \Delta(\mathcal{U}_\omega(\C))$. We need the following fact about the diametral  dimension.
\begin{lemma}\label{closed-her} (cf.\ \cite[Theorem 1.6.2.4]{Dubinsky})
Let $E$ and $F$ be  nuclear Fr\'echet  spaces and suppose that  $F$ is  isomorphic to a closed subspace of $E$. Then, $\Delta(E) \subseteq \Delta(F)$.
\end{lemma}
The proof of the next result is inspired by the proof of \cite[Theorem 3.2]{Langenbruch2016}.
\begin{proposition}\label{lower-estimate}
Let $\omega$ be a weight function satisfying \eqref{translation-invariant}. Then, $ \Lambda'_\infty(\omega^*(n)) \subseteq \Delta(\mathcal{U}_\omega(\C))$.
\end{proposition}
\begin{proof} We define  $E$ as the Fr\'echet space consisting of all $(c_{n,j})_{(n,j) \in \N \times \Z} \in \C^{\N \times \Z}$ such that  for all $r > 0$
$$
\sup_{(n,j) \in \N \times \Z} |c_{n,j}| e^{r(n+\omega(j))}  < \infty.
$$
We claim that $E$ contains a closed subspace that is isomorphic to $\mathcal{U}_\omega(\C)$. Before we prove the claim, let us show how it entails the result. Note that $E \cong \Lambda_\infty(\beta)$, where $\beta$ is the increasing rearrangement of the set $\{ n + \omega(j) \, | \, n \in \N, j \in \Z\}$. In the proof of \cite[Theorem 3.2(b)]{Langenbruch2016} it is shown that $\omega^*(n) = O(\beta_n)$. In particular, as $\log t = o(\omega^*(t))$, $E$ is nuclear. Hence, the claim and Lemma \ref{closed-her} yield that
$$
\Lambda'_\infty(\omega^*(n)) \subseteq \Lambda'_\infty(\beta) = \Delta(\Lambda_\infty(\beta)) = \Delta(E) \subseteq \Delta(\mathcal{U}_\omega(\C)).
$$
We now show the claim. Since $\Lambda_\infty(n) \cong \mathcal{O}(\C)$, $E$ is isomorphic to the Fr\'echet space $E_0$ consisting of all $(\varphi_j)_{j \in \Z} \in \mathcal{O}(\C)^\Z$ such that for all $h > 0$
$$
\sup_{j \in \Z} \sup_{|z| \leq h} |\varphi_j(z)|e^{h\omega(j)} < \infty.
$$
The claim now follows by noting that the mapping
$$
\mathcal{U}_\omega(\C) \rightarrow E_0 \, : \, \varphi \mapsto (\varphi(\,\cdot + j))_{j \in \Z}
$$
is a topological embedding. 
\end{proof}
Next, we show that  $\Delta(\mathcal{U}_\omega(\C)) \subseteq \Lambda'_\infty(\omega^*(n))$. To this end, we will use the generalized diametral dimension and   an extension of the condition $(\Omega)$. Both these notions were  introduced in \cite{Langenbruch1987}. The \emph{generalized diametral dimension} of a Fr\'echet space $E$ with a fundamental decreasing sequence $(U_n)_{n \in \N}$ of neighbourhoods of zero is defined  as the set
$$
\widetilde{\Delta}(E) := \{  (c_n)_{n \in \N} \in \C^\N \, | \, \forall m \in \N \, \exists k \geq m \, \exists L > 0 : \, \lim_{n \to \infty} c_n \delta_n(U_k,U_m)^L = 0 \}.
$$
We have that $\widetilde{\Delta}(\Lambda_\infty(\beta)) = \Lambda'_\infty(\beta)$ for any nuclear power series space $\Lambda_\infty(\beta)$ of infinite type with  \cite[Equality (1.4)]{Langenbruch1987}. 

Let $E$ and $\widetilde{E}$ be Fr\'echet spaces with  fundamental decreasing sequences $(U_n)_{n \in \N}$ and $(\widetilde{U}_n)_{n \in \N}$ of neighbourhoods of zero, respectively. Let $d: E \rightarrow \widetilde{E}$ be a continuous linear mapping. The triple $(d,E,\widetilde{E})$ is said to satisfy $(\Omega)$  if
$$
\forall n \in \N \, \exists m \geq n \,  \forall k \in \N \, \exists C,L > 0 \, \forall r > 0 \, : \, 
\widetilde{U}_m \subseteq e^{-r} \widetilde{U}_n + Ce^{Lr} d(U_k).
$$
Note that $E$ satisfies $(\Omega)$ if and only if $(\operatorname{id}_{E},E, E)$ satisfies $(\Omega)$.  The next result is a substitute for Lemma \ref{closed-her}.
\begin{lemma}\label{dd-omega} \emph{\cite[Lemma 1.3(a)]{Langenbruch1987}}
Let $E$ and $\widetilde{E}$ be Fr\'echet spaces and let $d: E \rightarrow \widetilde{E}$ be a continuous linear mapping. If $(d,E,\widetilde{E})$ satisfies $(\Omega)$, then $\Delta(E) \subseteq \widetilde{\Delta}(\widetilde{E})$.
\end{lemma}
The following observation generalizes the fact that condition $(\Omega)$ for single Fr\'echet spaces is inherited to quotient spaces.
\begin{lemma}\label{quotient-ext}
Let $E_i$ and $\widetilde{E}_i$ be Fr\'echet spaces and let $d_i: E_i \rightarrow \widetilde{E}_i$ be a continuous linear mapping for $i = 1,2$. Let $T: E_1 \rightarrow E_2$ and $\widetilde{T}: \widetilde{E}_1 \rightarrow \widetilde{E}_2$ be surjective continuous linear mappings such that $\widetilde{T} \circ d_1 = d_2 \circ T$. If $(d_1,E_1,\widetilde{E}_1)$ satisfies $(\Omega)$, then so does $(d_2,E_2,\widetilde{E}_2)$.
\end{lemma}
\begin{proof}
The proof is straightforward and therefore left to the reader.
\end{proof}
\begin{proposition}\label{upper-estimate}
Let $\omega$ be a weight function satisfying $(\alpha)$. Then, $\Delta(\mathcal{U}_\omega(\C)) \subseteq \Lambda'_\infty(\omega^*(n))$.
\end{proposition}
\begin{proof} Recall that, since $\omega$ satisfies $(\alpha)$, $\mathcal{A}_{\omega}(\R)$  coincides with the space $\mathcal{H}^\infty_{\omega}(\R)$ from \cite[p.\ 91]{Langenbruch2016}.
In \cite[Section 4]{Langenbruch2016} (see particularly the proof of \cite[Theorem 4.6]{Langenbruch2016}) it is shown that $\widetilde{\Delta}(\mathcal{A}'_\omega(\R)) \subseteq \Lambda'_\infty(\omega^*(n))$. Let $d: \mathcal{U}_\omega(\C) \rightarrow \mathcal{A}'_{\omega}(\R)$ be the inclusion mapping (cf.\ the discussion before Proposition \ref{STFT-dual-2}).
By Lemma \ref{dd-omega}, it suffices to show that $(d, \mathcal{U}_\omega(\C), \mathcal{A}'_{\omega}(\R))$ satisfies $(\Omega)$. We shall achieve this by combining Lemma \ref{quotient-ext} with the results from Subsection \ref{STFT}.  Set $E_1 = C_\omega(\R^2)$, $\widetilde{E}_1 =  \widetilde{C}_\omega(\R^2)$, and let $d_1:  C_\omega(\R^2) \rightarrow  \widetilde{C}_\omega(\R^2)$ be the inclusion mapping.  Set  $E_2 =  \mathcal{U}_\omega(\C)$, $\widetilde{E}_2 = \mathcal{A}'_{\omega}(\R)$ and $d_2 = d$. Next, fix a non-zero window $\psi \in \mathcal{U}_{\omega}(\C)$ and consider  the linear mappings $T = V^*_\psi : C_\omega(\R^2) \rightarrow  \mathcal{U}_\omega(\C)$ and $\widetilde{T} = V^*_\psi : \widetilde{C}_\omega(\R^2) \rightarrow  \mathcal{A}'_\omega(\R)$. The mappings $T$ and $\widetilde{T}$ are continuous and surjective by Proposition \ref{STFT-test} (and \eqref{eq:reconstructSTFT}) and Proposition \ref{STFT-dual-2}, respectively. Moreover, it is clear that $\widetilde{T} \circ d_1 = d_2 \circ T$. Hence, by Lemma \ref{quotient-ext}, it is enough to prove that $(d_1, C_\omega(\R^2),\widetilde{C}_\omega(\R^2))$ satisfies $(\Omega)$. This can be done by using a standard argument involving cut-off functions.
\end{proof}
We are ready to show Theorem \ref{main}.
\begin{proof}[Proof of Theorem \ref{main}]
 Theorem \ref{main-1} yields that $\mathcal{U}_{\omega}(\C) \cong \Lambda_{\infty}(\beta)$ for some nuclear power series space $\Lambda_{\infty}(\beta)$ of infinite type. By Proposition \ref{lower-estimate} and Proposition \ref{upper-estimate}, we have that
 $$
  \Lambda' _{\infty}(\beta) = \Delta(\Lambda_{\infty}(\beta)) = \Delta (\mathcal{U}_{\omega}(\C)) =  \Lambda'_\infty(\omega^*(n)).
 $$
Hence, the sequences $\beta$ and  $(\omega^*(n))_{n \in \N}$ are $O$-equivalent and, thus, $\mathcal{U}_{\omega}(\C) \cong \Lambda_{\infty}(\beta) \cong  \Lambda_\infty(\omega^*(n))$.
\end{proof}
\begin{corollary}
Let $\omega$ be a weight function satisfying $(\alpha)$. Then, $\mathcal{U}_\omega(\C)$ is stable, i.e., $\mathcal{U}_\omega(\C) \cong \mathcal{U}_\omega(\C) \times \mathcal{U}_\omega(\C)$.
\end{corollary}
\begin{proof}
By Theorem \ref{main}, it suffices to show that $\Lambda_\infty(\omega^*(n))$ is stable. This follows from the fact that $\omega^*$ satisfies $(\alpha)$.
\end{proof}
We have the following open problem.
\begin{problem} \label{open}
Let $\omega$ be a weight function satisfying $(A)_\infty$. Show that $\mathcal{U}_{\omega}(\C) \cong  \Lambda_\infty(\omega^*(n))$.
\end{problem}
Let $\omega$ be a weight function satisfying $(A)_\infty$.  Problem \ref{open} is equivalent to $\mathcal{U}_{\omega}(\C) \cong  \Lambda_\infty(\beta)$ for some power series space $\Lambda_{\infty}(\beta)$ of infinite type and $\Delta (\mathcal{U}_{\omega}(\C)) =  \Lambda'_\infty(\omega^*(n))$ (cf.\ the proof of Theorem \ref{main}). The former is shown in Theorem \ref{main-1} and the inclusion $ \Lambda'_\infty(\omega^*(n)) \subseteq \Delta (\mathcal{U}_{\omega}(\C)) $ is proven in Proposition \ref{lower-estimate}.
However, we do not know how to show that $\Delta(\mathcal{U}_\omega(\C)) \subseteq \Lambda'_\infty(\omega^*(n))$ without assuming that $\omega$ satisfies $(\alpha)$.

\subsection{Examples} We now determine up to $O$-equivalence the function $\omega^*$ for various weight functions $\omega$; see also \cite[Section 5]{Langenbruch2016}.  Given two weight functions $\omega$ and $\eta$, we write $\omega \asymp \eta$ to indicate that $\omega = O(\eta)$ and $\eta = O(\omega)$. 
\begin{itemize}
\item[$(i)$] $\omega(t) = t^{1/\nu}$, where $\nu > 0$. Then, $$\omega^*(t) \asymp t^{1/(\nu+1)}.$$
\item[$(ii)$] $\omega(t) = (\log t)^{a_1} (\log \log t)^{a_2}$ for $t$ large enough, where  $a_1 > 1$ and $a_2 \in \R$ or $a_1 = 1$ and $a_2 > 0$. Then,  
$$\omega^*(t) \asymp (\log t)^{a_1} (\log \log t)^{a_2}.$$
\item[$(iii)$] $\omega(t) = t^{1/\nu} (\log t)^{a}$ for $t$ large enough, where  $\nu > 0$ and $a \in \R$. Then, 
$$
\omega^*(t) \asymp t^{1/(\nu+1)} (\log t)^{a\nu/(\nu+1)}.
$$
\item[$(iv)$] $\omega(t) = e^{(\log t)^a}$, where $ 0 < a < 1$.  We will use some results from the theory of regularly varying functions \cite{BGT} to determine  $\omega^*$. By \cite[Proposition 1.5.15]{BGT}, the inverse of the function $t\omega(t)$ is asymptotically equivalent to $t\omega^\#(t)$, where $\omega^\#$ is the de Bruijn conjugate of $\omega$ \cite[p.\ 29]{BGT}. Hence, by \eqref{l-repr}, $\omega^*$ is asymptotically equivalent to $1/\omega^\#$. The function $\omega^\#$ is determined in \cite[p.\ 435, Example 3]{BGT}. From this result, we obtain that 
\begin{gather*}
\omega^*(t) \asymp  e^{(\log t)^a}, \qquad \mbox{if $a < 1/2$}, \\
\omega^*(t) \asymp  e^{(\log t)^a - a(\log t)^{2a-1}}, \qquad \mbox{if $1/2 \leq a < 2/3$}, \\
\omega^*(t) \asymp   e^{(\log t)^a - a(\log t)^{2a-1} + \frac{a}{2}(3a-1)(\log t)^{3a-2}}, \qquad \mbox{if $2/3 \leq a < 3/4$},
\end{gather*}
\end{itemize}
and so on.

Finally, we apply Theorem \ref{main} to obtain sequence space representations of the projective Gelfand-Shilov spaces $\Sigma^1_\nu$ and $\Sigma^\nu_1$. Given $\mu, \nu > 0$, we define $\Sigma^\mu_\nu$ as the Fr\'echet space consisting of all $\varphi \in C^\infty(\R)$ such that for all $h > 0$
$$
\sup_{p,q \in \N}\sup_{x \in \R} \frac{h^{p+q} |\varphi^{(p)}(x) x^q| }{p!^\mu q!^\nu} < \infty.
$$
The spaces $\Sigma^\mu_\nu$ are the projective counterparts of the classical Gelfand-Shilov spaces $\mathcal{S}^\mu_\nu$ \cite[Chapter IV]{GS}.
\begin{theorem}\label{main-example}
The spaces $\Sigma^1_\nu$ and $\Sigma^\nu_1$ are both isomorphic to $\Lambda_\infty( n^{1/(\nu+1)})$ for $\nu > 0$.
\end{theorem}
\begin{proof}
Since  the Fourier transform is an isomorphism between $\Sigma^1_\nu$ and $\Sigma^\nu_1$ (cf.\ \cite[Chapter IV, Section 6]{GS}), it suffices to show that  $\Sigma^1_\nu$ is  isomorphic to $\Lambda_\infty( n^{1/(\nu+1)})$. We have that $\Sigma^1_\nu = \mathcal{U}_{t^{1/\nu}}(\C)$ (cf.\ \cite[Chapter IV, Section 2]{GS}).  Hence, the result follows from Theorem \ref{main} and the above example $(i)$.
\end{proof}
\noindent {\bf Acknowledgements.} The author would like to thank the anonymous referee for pointing out Theorem \ref{basis}.

\end{document}